\documentclass[12pt]{amsart}
\usepackage[utf8]{inputenc}
\usepackage[margin=1in]{geometry}

\usepackage[
    backend=biber,
    style=ieee,
    citestyle=numeric-comp,
    sorting=nyt,
    dashed=false
  ]{biblatex}
\usepackage{hyperref}
\usepackage[hyphenbreaks]{breakurl}
\usepackage{xurl}
\hypersetup{
    colorlinks=true,
    linkcolor=blue,
    citecolor=red,
    breaklinks=true
}

\usepackage{graphicx}
\usepackage{amsmath}
\usepackage{amssymb}
\usepackage{appendix}
\usepackage{amsthm}
\usepackage{booktabs}
\usepackage{graphicx}
\usepackage{subcaption}
\usepackage{comment}
\usepackage{enumerate}
\usepackage{esint}

\usepackage[dvipsnames]{xcolor}

\newtheorem{theorem}{Theorem}[section]
\newtheorem{corollary}{Corollary}[theorem]
\newtheorem{prop}[theorem]{Proposition}
\newtheorem{lemma}[theorem]{Lemma}

\theoremstyle{definition}
\newtheorem{definition}[theorem]{Definition}
\theoremstyle{remark}

\usepackage{amsmath}
\usepackage{amssymb}
\usepackage{mathrsfs}

\usepackage{multicol}

\newcommand{\R}{{\mathbb R}}

\newcommand{\F}{{\mathcal{F}}}

\newcommand{\sS}{{\mathscr{S}}}

\usepackage{todonotes}


\title[Compactness of Fourier Concentration operators]{Compactness of Fourier concentration operators}

\author{Helge J. Samuelsen}
\address{Department of Mathematical Sciences,
         Norwegian University of Science and Technology,
         Trondheim, Norway}

\email{helge.j.samuelsen@ntnu.no}
\date{\today}

\begin{document}

\begin{abstract}
We present a sufficient condition on sets $E$ and $F$ in $\R^d$ to ensure compactness of Fourier concentration operators by introducing the notion of sets which are very thin at infinity. We are able to show that if the sets $E$ and $F$ are both very thin at infinity, then the associated Fourier concentration operator is compact on $L^2(\R^d)$. The proof relies on a combination of the Logvinenko-Sereda uncertainty principle together with an uncertainty principle due to Shubin, Vakilian and Wolff. This provides a partial answer to a question posed by Katsnelson and Machluf on truncated Fourier operators.
\end{abstract}
\maketitle

\section{Introduction}\label{sec:Intro}
Let $E$ and $F$ be subsets of $\R^d$. We consider the projection $P_E$ and Fourier projection $Q_F$ onto the sets $E$ and $F$ defined by
\[
P_Ef=\chi_E f, \qquad \widehat{Q_Ff}=\chi_F\widehat{f},
\]
for functions $f$ in the Schwartz class $\sS(\R^d)$. Here the Fourier transform is normalised by
\[
\mathcal{F}(f)(\xi)=\widehat{f}(\xi):=\int_{\R^d}e^{-2\pi i \xi\cdot x}f(x)\,dx.
\]
Composing these projections gives a bounded operator on $L^2(\R^d)$ defined by
\[
Q_FP_E f(x):=\F^{-1}(\chi_F\F(\chi_Ef))(x)=\iint_{\R^{2d}}e^{-2\pi i\xi\cdot(y-x)}\chi_F(\xi)\chi_E(y)f(y)\,dy\,d\xi,
\]
for $f\in \sS(\R^d)$. We call the operator $Q_FP_E$ a Fourier concentration operator, and Plancherel's theorem tells us that this operator is bounded by one in the operator norm for any pair of sets $E,F\subseteq \R^d$. 

The operator $Q_FP_E$ dates back to a series of Bell lab papers from the 1960's by Landau, Slepian and Pollak \cite{Pollak1,Pollak2,Pollak3,Slepian4}. The eigenvalues of the self-adjoint Fourier concentration operator
\[
(Q_FP_E)^*Q_FP_E=P_EQ_FP_E,
\]
were investigated by many authors, we refer the reader to the recent results in \cite{Kulikov, Romero} and the references there in. Common for these papers is that the subsets $E$ and $F$ have finite Lebesgue measure. If the sets $E$ and $F$ have finite Lebesgue measure, then the operator $Q_FP_E$ can be written as an integral operator 
\[
Q_FP_Ef(x)=\int_{\R^d}K(x,y)f(y)\,dy,
\]
with integral kernel $K$ given by
\[
K(x,y)=\int_{\R^{d}}e^{-2\pi i\xi\cdot(y-x)}\chi_F(\xi)\chi_E(y)\,d\xi=\widehat{\chi}_F(y-x)\chi_E(y)\in L^2(\R^{2d}).
\]
Since the integral kernel is in $L^2(\R^d)$, it is a Hilbert-Schmidt operator and thus a compact operator on $L^2(\R^d)$ \cite{Folland_Uncertainty}. The reverse statement is also true. If the Fourier concentration operator is in the Hilbert-Schmidt class, then the sets $E$ and $F$ have finite measure. For more general sets, the question of compactness is not well understood \cite{Katsnelson}. The main goal of this paper is to provide a sufficient condition on the sets $E$ and $F$ to ensure compactness of the Fourier concentration operator. This gives a partial answer to a question posed by Katsnelson and Machluf \cite{Katsnelson} on the compactness of the truncated Fourier operator $\mathcal{F}_E=\F\circ Q_EP_E$.

We define the radial function $\rho:\R^d\to (0,1]$ by
\[
\rho(x)=\min\left\{1,\frac{1}{|x|}\right\}.
\]
The radial function $\rho$ was used by Shubin, Vakilian and Wolff in \cite{Wolff_Uncertainty} to derive an uncertainty result for so-called $\varepsilon$-thin sets. A set $E$ is called $\varepsilon$\textit{-thin} if $|E\cap B(x,\rho(x))|< \varepsilon |B(x,\rho(x))|$ for some $\varepsilon>0$ and all $x\in\R^{d}$. Using a clever decomposition of the identity operator, they proved the following uncertainty principle.
 \begin{theorem}[\cite{Wolff_Uncertainty}, Thm $2.1$]\label{thm:Wolff_Unc}
There are $\varepsilon>0$ and $C<\infty$ such that if $E,F\subseteq \R^d$ are $\varepsilon$-thin sets, then for any $f\in L^2(\R^d)$
\[
\|f\|_{L^2}\leq C\left(\|f\|_{L^2(E^C)}+\|\widehat{f}\|_{L^2(F^{C})}\right).
\] 
\end{theorem}
An extension of this result can be found in \cite{Kovrizhkin1}. Theorem \ref{thm:Wolff_Unc} tells us that the only function on $L^2(\R^d)$ which is localised on $\varepsilon$-thin sets both in physical and Fourier space is the trivial function $f=0$.

We will slightly modify the definition of $\varepsilon$-thin sets in order to have a more quantitive control over the measure at infinity. This gives the following definition.
\begin{definition}[Very thin at infinity]
We say that $E\subseteq \R^d$ is \textit{very thin at infinity} if 
\begin{equation}\label{eq:defVTAI}
\lim_{|x|\to\infty}\frac{\left|E\cap B\big(x,\rho(x)\big)\right|}{\left|B\big(x,\rho(x)\big)\right|}=0.
\end{equation}
\end{definition}
The main result is that very thin at infinity is a sufficient condition on the sets $E$ and $F$ for the associated Fourier concentration operator to be a compact operator on $L^2(\R^d)$.
\begin{theorem}\label{thm:MainThm}
Let $E$ and $F$ be subsets of $\R^d$. If both the subsets $E$ and $F$ are very thin at infinity, then the operator $Q_FP_E$ is a compact operator on $L^2(\R^d)$.
\end{theorem}
Any bounded set $E$ is necessarily very thin at infinity as the intersection $E\cap B(x,\rho(x))$ is empty as $|x|$ tends to infinity. This means that Theorem \ref{thm:MainThm} recovers the compactness of the Fourier concentration operator for bounded subsets. However, sets of finite Lebesgue measure do not need to be very thin at infinity. Let $\{x_n\}_{n\in\mathbb{N}}\subseteq \R^d$ be any sequence of points such that $|x_n|=n^2$, and consider the set
\[
E=\bigcup_{n\in\mathbb{N}}B\left(x_n,\rho(x_n)\right).
\]
The set $E$ has finite Lebesgue measure since
\[
|E|=\sum_{n\in\mathbb{N}}|B\left(x_n,\rho(x_n))\right|=\sum_{n=\mathbb{N}}\frac{|B(0,1)|}{n^{2d}}<\infty,
\]
and consequently the operator $Q_EP_E$ is a Hilbert-Schmidt operator on $L^2(\R^d)$. Nevertheless, the set $E$ does not satisfy \eqref{eq:defVTAI} since for each $n\in\mathbb{N}$ we have
\[
\frac{\left|E\cap B(x_n,\rho(x_n))\right|}{\left|B(x_n,\rho(x_n))\right|}=1.
\]
We can therefore conclude that Theorem \ref{thm:MainThm} provides a sufficient, but not necessary condition for the compactness of Fourier concentration operators.

There are examples of sets with infinite Lebesgue measure that are very thin at infinity. For simplicity we restrict our discussion to $\R$. Let $\{x_{n}\}_{n\in\mathbb{N}}\subseteq \R$ be a sequence with $|x_{n}|=n+1$. Define the set
\[
E=\bigcup_{n\in\mathbb{N}}B\left(x_{n},\eta(x_{n})\right),
\]
where $\eta:\R\to (0,1]$ is the function given by
\[
\eta(x)=\min\left\{1,\frac{1}{|x|\log|x|}\right\}.
\]
The set $E$ has infinite Lebesgue measure since
\[
|E|= \sum_{n=2}^\infty\frac{|B(0,1)|}{n\log n}=\infty,
\]
while for any $R>1$ we have
\[
\sup_{|x|>R}\frac{|E\cap B(x,\rho(x))|}{|B(x,\rho(x))|}\leq \frac{1}{\log R}\xrightarrow{R\to \infty}{0},
\]
as $\eta(x)\leq \rho(x)$ for all $x\in \R$. This shows that $E$ is very thin at infinity, and that Theorem \ref{thm:MainThm} provides new sets for which the Fourier concentration operator is compact on $L^2(\R)$.


Compactness of the Fourier concentration operator is determined by the structure of $E$ and $F$ at infinity, as any bounded sets give rise to a Hilbert-Schmidt operator. We therefore decompose $E$ and $F$ as
\[
E=E^R\cup E_\infty^R,\qquad F=F^R\cup F_\infty^R
\] 
using the notation
\[
A^R=A\cap B(0,R),\qquad A_\infty^R=A\backslash A_R.
\]
The Fourier concentration operator can then be written as
\begin{equation}\label{eq:OpDecomp}
Q_FP_E=Q_{F^R}P_{E^R}+Q_{F_\infty^R}P_{E^R}+Q_{F^R}P_{E_\infty^R}+Q_{F_\infty^R}P_{E_\infty^R}.
\end{equation}
The operator $Q_{F^R}P_{E^R}$ is a compact operator as $F^R$ and $E^R$ are sets of finite Lebesgue measure. Since the set of compact operators is a closed operator ideal in the space of bounded linear operators on $L^2(\R^d)$, it suffices to prove that the three remaining operators are small in operator norm as $R$ tends to infinity. This is done through two different uncertainty principles. The remaining part of this section is devoted to the explanation on how the uncertainty principles are used in the proof of Theorem \ref{thm:MainThm}. A formal proof of Theorem \ref{thm:MainThm} will be presented towards the end of the paper.

There are two different uncertainty principles present in the proof of Theorem \ref{thm:MainThm}. The first is that of Shubin, Vakilian and Wolff, Theorem \ref{thm:Wolff_Unc}. More specifically we exploit the techniques used in the proof of Theorem \ref{thm:Wolff_Unc}. Since the  sets $E$ and $F$ can be reduced to $\varepsilon$-thin sets by removing a sufficiently large ball centred at the origin, one can show for each $\varepsilon>0$ that 
\[
\|Q_{F_\infty^R}P_{E_\infty^R}\|_{L^2\to L^2}<\frac{\varepsilon}{3},
\]
for all $R>0$ sufficiently large. The estimate relies on the decomposition of the identity as introduced in \cite{Wolff_Uncertainty}.

The operator norm of the cross-term $Q_{F_\infty^R}P_{E^R}$ can be bounded by 
\[
\|Q_{F^R}P_{E_\infty^R}\|_{L^2\to L^2}\leq \|Q_{B(0,R)}P_{E_\infty^R}\|_{L^2\to L^2},
\]
as $F^R\subseteq B(0,R)$. This is a consequence of the monotonicity of the Lebesgue integral. Since the measure of $E_\infty^R\cap B(x,1)$ tends to zero as $R$ tends to infinity, it is possible to conclude that $\|Q_{F^R}P_{E_\infty^R}\|_{L^2\to L^2}<\varepsilon/3$ by an application of the Logvinenko-Sereda uncertainty principle, see \cite[Prop $10.6$]{Schlag}. For a more general version of the Logvinenko-Sereda uncertainty principle, we refer to \cite{Kovrizhkin2}. The other cross-term operator $Q_{F_\infty^R}P_{E^R}$ can also be bounded on $L^2(\R^d)$ in the same fashion by using
\[
\overline{\mathcal{F}(Q_FP_Ef)}=P_FQ_{E}\overline{\mathcal{F}(f)},\qquad (Q_FP_E)^*f=P_EQ_Ff,
\]
to conclude that the order in which the projection operators are composed does not influence the operator norm.

\section{Uncertainty Principles and proof of Theorem \ref{thm:MainThm}}

\subsection{Decomposition of the identity}
We begin by presenting a decomposition of the identity operator on $L^2(\R^d)$ used by Shubin, Vakilian and Wolff to prove Theorem \ref{thm:Wolff_Unc}. The construction of the decomposition can be found in \cite{Wolff_Uncertainty,Kovrizhkin1}, but for completeness we also include it here.

Let $\varphi\in \sS(\R^d)$ be a radial function with $0\leq\widehat{\varphi}\leq 1$, $\mathrm{supp }(\widehat{\varphi})\subseteq B(0,2)$ and $\widehat{\varphi}\equiv 1$ on $B(0,1)$. For each $j\in\mathbb{N}$, define $\varphi_j(x)=2^{jd}\varphi(2^jx)$ such that $\|\varphi_j\|_{L^1}=\|\varphi\|_{L^1}$. Moreover, $\widehat{\varphi_j}(\xi)=\widehat{\varphi}(2^{-j}\xi)$, and thus $\mathrm{supp }(\widehat{\varphi_j})\subseteq B(0,2^{j+1})$ and $\widehat{\varphi_j}\equiv 1$ on $B(0,2^j)$.
From this we construct a Littlewood-Paley decomposition of $\R^d$. Denote by $\psi_0=\widehat{\varphi}$ and $\psi_j(x)=\widehat{\varphi}(2^{j+1}x)-\widehat{\varphi}(2^{j}x)$ such that $\mathrm{supp }(\psi_j)\subseteq B(0,2^{j+1})\backslash B(0,2^{j-1})$ for each $j\in\mathbb{N}$. From this, we define the operators
\begin{equation}\label{eq:DefOp}
Sf=\sum_{j=0}^\infty \psi_j\left(\varphi_j*f\right),\qquad Tf=\sum_{j=0}^{\infty}\psi_j\left(f-\varphi_j*f\right).
\end{equation}
Each sum in \eqref{eq:DefOp} converges pointwise as there are at most three non-zero terms for each $x\in \R^d$, and $T+S=\mathrm{Id}_{L^2}$ as $\{\psi_j\}$ is a partition of unity of $\R^d$. We also need the following lemma.
\begin{lemma}[\cite{Wolff_Uncertainty}, Lem. $2.2$]\label{lem:SchurTest}
Let $\varphi_j$ and $\psi_j$ be defined as above. Define the kernels
\[
A(x,y)=\sum_{j=0}^\infty \psi_j(x)\varphi_j(x-y),\qquad B(\xi,\eta)=\sum_{j=0}^\infty \widehat{\psi_j}(\xi-\eta)\left(1-\widehat{\varphi_j}(\eta)\right).
\]
There exist $C>0$ such that for any $\varepsilon$-thin sets $E,F$ 
\begin{multicols}{2}
\begin{enumerate}[i)]
\item $\int_{\R^d}|A(x,y)|\,dy\leq  C$ for all $x$,
\item $\int_{\R^d}|A(x,y)|\,dx\leq  C$ for all $y$,
\item $\int_{E}|A(x,y)|\,dy\leq  C\varepsilon$ for all $x$,
\item $\int_{\R^d}|B(\xi,\eta)|\,d\eta\leq  C$ for all $\xi$,
\item $\int_{\R^d}|B(\xi,\eta)|\,d\xi\leq C$ for all $\eta$,
\item $\int_{F}|B(\xi,\eta)|\,d\xi\leq C\varepsilon$ for all $\eta$.
\end{enumerate}
\end{multicols}
\end{lemma}
The functions $A$ and $B$ in Lemma \ref{lem:SchurTest} are defined such that
\[
Sf(x)=\int_{\R^d} A(x,y)f(y)\,dy,\qquad \widehat{Tf}(\xi)=\int_{\R^{d}}B(\xi,\eta)\widehat{f}(\eta)\,d\eta,
\] 
while the integral estimates ensure that both $S$ and $T$ are bounded operators on $L^2(\R^d)$ by Schur's test. Moreover, for each pair of $\varepsilon$-thin sets $E$ and $F$, we have the norm bounds
\begin{equation}\label{eq:epsilon_est}
\|S(\chi_Ef)\|_{L^2}\leq C\sqrt{\varepsilon}\|f\|_{L^2},\qquad \|\chi_F \widehat{Tf}\|_{L^2}\leq C\sqrt{\varepsilon}\|f\|_{L^2},
\end{equation}
for every $f\in L^2(\R^d)$. With this in mind, we present the next estimate.
\begin{prop}\label{prop:WolffOpEst}
Let $E$ and $F$ be very thin at infinity. For each $\varepsilon>0$ there exists $R_0>1$ such that
\[
\|Q_{F_\infty^R}P_{E_\infty^R}\|_{L^2\to L^2}\leq \varepsilon,
\]
for every $R>R_0$.
\end{prop}
\begin{proof}
The proof is similar to that in \cite{Wolff_Uncertainty}. Since $E$ and $F$ are very thin at infinity, there exists $R_0$ such that $E_\infty^R,F_\infty^R$ are $\varepsilon$-thin in the definition of Shubin, Vakilian and Wolff for $R>R_0$. Namely, there exists $R_0>1$ such that 
\[
\sup_{|x|>R_0-1}\frac{|E \cap B(x,\rho(x))|}{|B(x,\rho(x))|}<\varepsilon, 
\]
as $E$ is very thin at infinity. Thus, for any $|x|\leq R_0-1$ it follows that
\[
\left|E_\infty^R\cap B\left(x,\rho(x)\right)\right|=0< \varepsilon \left|B\left(x,\rho(x)\right)\right|,
\]
as the intersection is empty whenever $R>R_0$, while 
\[
\left|E_\infty^R\cap B\left(x,\rho(x)\right)\right|\leq \left|E\cap B\left(x,\rho(x)\right)\right|< \varepsilon \left|B\left(x,\rho(x)\right)\right|,
\]
for $|x|>R_0-1$.

For any sets $A,B\subseteq \R^d$ and $g\in L^2(\R^d)$, it is possible to decompose the Fourier transform of $g$ as
\[
\widehat{g}=\widehat{Sg}+\widehat{Tg}=\widehat{S(\chi_Ag)}+\widehat{S(\chi_{A^C}g)}+\chi_B\widehat{Tg}+\chi_{B^C}\widehat{Tg},
\]
where $S$ and $T$ are the operators defined in \eqref{eq:DefOp}.
The last equality follows from linearity of $S$ and $\chi_A+\chi_{A^C}=1$.
If $A,B$ are $\varepsilon$-thin, then the triangle inequality and \eqref{eq:epsilon_est} yield
\[
\|\widehat{g}-\widehat{S(\chi_{A^C}g)}-\chi_{B^C}\widehat{Tg}\|_{L^2}\leq \|\widehat{S(\chi_Ag)}\|_{L^2}+\|\chi_B\widehat{Tg}\|_{L^2}\leq C\sqrt{\varepsilon}\|g\|_{L^2}.
\]
In particular, for any $g\in L^2(\R^d)$, we have
\begin{equation}\label{eq:epsilonEstimate}
\|\widehat{g}\|_{L^2(B)}=\|\widehat{g}-\chi_{B^C}\widehat{Tg}\|_{L^2(B)}\leq\|\widehat{g}-\chi_{B^C}\widehat{Tg}\|_{L^2}\leq \|S(\chi_{A^C}g)\|_{L^2}+C\sqrt{\varepsilon}\|g\|_{L^2}.
\end{equation}

To conclude, we note that for any $f\in L^2(\R^d)$ we can choose $A=E_\infty^R$, $B=F_\infty^R$ and $g=\chi_{E^{C}_{R}}f$ in \eqref{eq:epsilonEstimate}. Combined with Plancherel's theorem, this gives
\[
\left\|Q_{F_\infty^R}P_{E_\infty^R}f\right\|_{L^2}=\|\widehat{\chi_Af}\|_{L^2(B)}\leq \|S(\chi_{A^C}\chi_{A}f)\|_{L^2}+C\sqrt{\varepsilon}\|f\|_{L^2(A)}\leq C\sqrt{\varepsilon}\|f\|_{L^2}.
\]
As $f\in L^2(\R^d)$ is arbitrarily chosen, this implies that 
\[
\left\|Q_{F_\infty^R}P_{E_\infty^R}\right\|_{L^2\to L^2}=\sup_{\|f\|_{L^2}=1}\left\|Q_{F_\infty^R}P_{E_\infty^R}f\right\|_{L^2}\leq C\sqrt{\varepsilon}.
\]
\end{proof}

\subsection{Logvinenko-Sereda Uncertainty Principle}
To control the cross-terms, we utilise the Logvinenko-Sereda uncertainty principle. More specifically, we apply the Logvinenko-Sereda uncertainty principle to estimate the norm of $Q_{B(0,R)}P_{E_\infty^R}$. Let us recall a simple version of the Logvinenko-Sereda theorem.
\begin{prop}[\cite{Schlag}, Prop $10.6$]\label{prop:Logvinenko-Sereda}
Let $\alpha>0$ and suppose $A\subseteq \R^d$ satisfies
\begin{equation}\label{eq:LogSerCond}
|A\cap B(x,1)|<\alpha |B(x,1)|
\end{equation}
for all $x\in \R^d$. If $f\in L^2(\R^d)$ satisfies $\mathrm{supp }(\widehat{f})\subseteq B(0,1)$, then
\[
\|f\|_{L^2(A)}\leq \delta(\alpha)\|f\|_{L^2},
\]
where $\delta(\alpha)\to0$ as $\alpha\to 0$.
\end{prop}

Since any function $f$ with $\mathrm{supp }(\widehat{f})\subseteq B(0,1)$ is an eigenfunction of the operator $Q_{B(0,1)}$, with eigenvalue $1$, it follows that the Logvinenko-Sereda uncertainty principle implies
\[
\|P_AQ_{B(0,1)}f\|_{L^2}\leq \delta(\alpha)\|Q_{B(0,1)}f\|_{L^2},
\]
whenever $A$ satisfies \eqref{eq:LogSerCond}. Thus, through a change of variables, it is possible to connect the operator $Q_{B(0,R)}P_{E_\infty^R}$ with an operator of the form $Q_{B(0,1)}P_{A}$ where the $R$ dependence is entirely on the set $A$ in physical space. 
\begin{lemma}
Let $A\subseteq \R^{2d}$. Then, for any $R>0$ and $x\in \R^d$
\[
Q_{B(0,R)}P_{A}f(x)=Q_{B(0,1)}P_{RA}f_{R^{-1}}\left(Rx\right),
\]
for every $f\in \sS(\R^d)$, and where $f_R(x)=f(Rx)$.
\end{lemma}
\begin{proof}
The proof follows by a change of variables in the definition of the operator $Q_{B(0,R)}P_A$. For any $f\in\sS(\R^d)$, we have
\begin{align*}
Q_{B(0,R)}P_{A}f(x)
=&\,\iint_{\R^{2d}}e^{-2\pi i \xi\cdot(y-x)}\chi_{B(0,R)}(\xi)\chi_{A}(y)f(y)\,dy\,d\xi\\
=&\, R^d\iint_{\R^{2d}}e^{-2\pi i R\eta\cdot(y-x)}\chi_{B(0,1)}(\eta)\chi_{A}(y)f(y)\,dy\,d\eta\\
=&\, \iint_{\R^{2d}}e^{-2\pi i \eta\cdot(t-Rx)}\chi_{B(0,1)}\chi_{RA}(t)f(R^{-1}t)\,dt\,d\eta\\
=&Q_{B(0,1)}P_{RA}f_{R^{-1}}(Rx).
\end{align*}
\end{proof}

\begin{corollary}\label{cor:BndChangeVar}
Let $A\subseteq\R^{d}$. Then, for any $R>0$,
\[
\|Q_{B(0,R)}P_{A}\|_{L^2\to L^2}=\|Q_{B(0,1)}P_{RA}\|_{L^2\to L^2}.
\]
\end{corollary}

\begin{proof}
For any $f\in L^2(\R^d)$ it follows that
\begin{align*}
\|Q_{B(0,R)}P_{A}f\|_{L^2}^2
=&\,\int_{\R^d}|Q_{B(0,1)}P_{RA}f_{R^{-1}}(Rx)|^2\,dx\\
=&\,R^{-d}\|Q_{B(0,1)}P_{RA}f_{R^{-1}}\|_{L^2}\\
\leq&\, R^{-d}\|Q_{B(0,1)}P_{RA}\|_{L^2\to L^2}^2\|f_{R^{-1}}\|_{L^2}^2\\
=&\|Q_{B(0,1)}P_{RA}\|_{L^2\to L^2}^2\|f\|_{L^2}^2,
\end{align*}
and so $\|Q_{B(0,R)}P_{A}\|_{L^2\to L^2}\leq\|Q_{B(0,1)}P_{RA}\|_{L^2\to L^2}$. Replacing $f$ with $f_R$ and $x$ with $R^{-1}x$ gives $\|Q_{B(0,1)}P_{RA}\|_{L^2\to L^2}\leq\|Q_{B(0,R)}P_{A}\|_{L^2\to L^2}$ by a similar calculation.
\end{proof}

The next lemma is what connects the operators $Q_{B(0,R)}P_{E_\infty^R}$ with the Logvinenko-Sereda uncertainty principle. It shows that for sufficiently large $R$, the set $RE_\infty^R$ satisfies the condition of the Logvinenko-Sereda principle.
\begin{lemma}\label{lem:BallEstimate}
Let $E\subseteq \R^d$ be very thin at infinity. For every $\varepsilon>0$ there exists $R_0>0$ and a universal constant $C_d>0$ such that for all 
\[
|RE_\infty^R\cap B(x,1)|\leq C_d\varepsilon,
\]
for all $R>R_0$ and $x\in\R^d$.
\end{lemma}
\begin{proof}
There exists $R_0>0$ such that
\begin{equation}\label{eq:EpsThinIneq}
|E_\infty^R\cap B(x,\rho(x))|<\varepsilon |B(x,\rho(x))|,
\end{equation}
for all $R>R_0$ since $E$ is very thin at infinity. 

We may assume that $|x|\geq R^2-1$ since $RE_\infty^R\cap B(x,1)$ is empty otherwise. The measure of the intersection can be written as
\begin{align}\label{eq:ChangOfVari}
|RE_\infty^R\cap B(x,1)|
=&R^d\left|E_\infty^R\cap B\left(\frac{x}{R},\frac{1}{R}\right)\right|,
\end{align}
through a change of variables.

For now, assume $|x|>R^2$. It then follows that
\[
\frac{1}{R}=\frac{R}{R^2}> \frac{R}{|x|}>\frac{R}{|x|+1}>\frac{R}{2|x|}.
\]
We can then find a finite collection $\mathcal{X}=\{x_k\}\subseteq B(R^{-1}x,R^{-1})$ such that
\[
B\left(\frac{x}{R},\frac{1}{R}\right)\subseteq \bigcup_{x_k\in \mathcal{X}} B\left(x_k,\frac{R}{2|x|}\right),
\]
and $\mathrm{Card}(\mathcal{X})\leq C_d (|x|/R^2)^d$ as we cover a ball of radius $R^{-1}$ by balls of radius $R/(2|x|)$. The radius $R/(2|x|)$ is chosen so that
\[
\frac{R}{2|x|}\leq \rho(x_k)\leq 2\frac{R}{|x|},
\]
for every $x_k\in \mathcal{X}$. Here we used that $|x|>R^2>1$ and $R^{-1}(|x|-1)<|x_k|<R^{-1}(|x|+1)$. Thus, by monotonicity and subadditivity of the Lebesgue measure, it follows from \eqref{eq:EpsThinIneq} that
\begin{align*}
|RE_\infty^R\cap B(x,1)|
\leq&\, R^d\sum_{x_k\in \mathcal{X}}|E_\infty^R\cap B(x_k,\rho(x_k)|\\
\leq &\,R^d\sum_{x_k\in\mathcal{X}}\varepsilon|B(x_k,\rho(x_k)|\\
\leq&\, \varepsilon R^d\mathrm{Card}(\mathcal{X})\left|B\left(0,2\frac{R}{|x|}\right)\right|\\
\leq &\, 2^dC_d \varepsilon.
\end{align*}

If $R^2-1\leq |x|\leq R^2$, then $R^{-1}\leq R/|x|=\rho(R^{-1}x)$, and by monotonicity and \eqref{eq:EpsThinIneq} we have
\begin{align*}
|RE_\infty^R\cap B(x,1)|
=&\,R^{d}\left|E_\infty^R\cap B\left(\frac{x}{R},\frac{1}{R}\right)\right|\\
\leq&\, R^d\left|E_\infty^R\cap B\left(\frac{x}{R},\rho\left(\frac{x}{R}\right)\right)\right|\\
\leq&\, \omega_d R^{2d}|x|^{-d}\varepsilon\\
\leq&\, C_d \varepsilon,
\end{align*}
where we used that $|x|^{-1}\leq (R^2-1)^{-1}\leq 2R^{-2}$ and $\omega_d=|B(0,1)|$. 
\end{proof}
\begin{corollary}\label{cor:LogSerUnc}
For any $\varepsilon>0$ there exists $R_0>0$ such that
\[
\|Q_{B(0,R)}P_{E_\infty^R}\|_{L^2\to L^2}\leq \varepsilon,
\]
for all $R>R_0$.
\end{corollary}
\begin{proof}
By Lemma \ref{lem:BallEstimate} and Proposition \ref{prop:Logvinenko-Sereda}, it follows that for each $\alpha>0$ there exists $R_0>0$ such that for each $R>R_0$
\begin{equation}\label{eq:CorEstimate}
\|f\|_{L^2(RE_\infty^R)}\leq \delta (\alpha)\|f\|_{L^2},
\end{equation}
for every $f\in L^2(\R^{2d})$ with $\mathrm{supp }(\widehat{f})\subseteq B(0,1)$. However, inequality \eqref{eq:CorEstimate} can be written as
\[
\|P_{RE_\infty^R}Q_{B(0,1)}f\|_{L^2}\leq \delta(\alpha)\|Q_{B(0,1)}f\|_{L^2}\leq \delta(\alpha)\|f\|_{L^2},
\]
for any $f\in L^2(\R^d)$. Choosing $\alpha$ such that $\delta(\alpha)\leq \varepsilon$ combined with Corollary \ref{cor:BndChangeVar} results in
\[
\|Q_{B(0,R)}P_{E_\infty^R}\|_{L^2\to L^2}=\|Q_{B(0,1)}P_{RE_\infty^R}\|_{L^2\to L^2}=\|P_{RE_\infty^R}Q_{B(0,1)}\|_{L^2\to L^2}\leq \varepsilon,
\]
which gives the desired outcome.
\end{proof}

We are now ready for proof of Theorem \ref{thm:MainThm} as described towards the end of section \ref{sec:Intro}.
\begin{proof}[Proof of Theorem \ref{thm:MainThm}]
To prove Theorem \ref{thm:MainThm} we use the decomposition \eqref{eq:OpDecomp} and the triangle inequality to write
\[
\|Q_FP_E-Q_{F^R}P_{E^R}\|_{L^2\to L^2}\leq\|Q_{F_\infty^R}P_{E^R}\|_{L^2\to L^2}+\|Q_{F^R}P_{E_\infty^R}\|_{L^2\to L^2}+\|Q_{F_\infty^R}P_{E_\infty^R}\|_{L^2\to L^2},
\]
for any $R>0$. For each $\varepsilon>0$, there exists $R_1>1$ such that
\[
\|Q_{F_\infty^R}P_{E_\infty^R}\|_{L^2\to L^2}\leq \frac{\varepsilon}{3},
\]
for every $R>R_1$ by Proposition \ref{prop:WolffOpEst}. From Corollary \ref{cor:LogSerUnc} there is some $R_2>0$ such that
\[
\|Q_{F^R}P_{E_\infty^R}\|_{L^2\to L^2}\leq \|Q_{B(0,R)}P_{E_\infty^R}\|_{L^2\to L^2}\leq \frac{\varepsilon}{3},
\]
for all $R>R_2$. Plancherel's theorem implies that
\[
\|Q_{F_\infty^R}P_{E^R}f\|_{L^2}=\|\widehat{Q_{F_\infty^R}P_{E^R}f}\|_{L^2}=\|P_{F_\infty^R}Q_{E^R}\widehat{f}\|_{L^2}.
\]
Since the operators $Q_F$ and $P_E$ are self-adjoint, it follows that
\[
\|Q_{F_\infty^R}P_{E^R}\|_{L^2\to L^2}=\|P_{F_\infty^R}Q_{E^R}\|_{L^2\to L^2}=\|Q_{E^R}P_{F_\infty^R}\|_{L^2\to L^2}< \frac{\varepsilon}{3},
\]
for each $R>R_3$ by Corollary \ref{cor:LogSerUnc}. Combining these three estimates, we can conclude that for any $\varepsilon>0$,
\[
\|Q_FP_E-Q_{F^R}P_{E^R}\|_{L^2\to L^2}\leq \varepsilon,
\]
for every $R>R_0=\max\{R_1,R_2,R_3\}$. This shows that $Q_FP_E$ can be approximated by compact operators, and is therefore a compact operator itself.
\end{proof}

\section*{Acknowledgements}
This work was conducted during a research stay at Stanford University in the Autumn of 2024, and the research stay was funded through the NSF grant DMS-2247185. The author is grateful for the wonderful hospitality of Stanford University during the stay, and would like to specially thank Professor Eugenia Malinnikova and Professor Franz Luef for their insightful discussions which made this paper possible. The author also appreciate the helpful comments given by Associate Professor Sigrid Grepstad on earlier versions of this manuscript.

\printbibliography

\end{document}